\documentclass{kms-b}

\newcommand{\publname}{Bull.~Korean Math.~Soc.}
\issueinfo{}
  {}
  {}
  {}
\pagespan{1}{}
\copyrightinfo{}
  {Korean Mathematical Society}

\usepackage{graphicx}
\allowdisplaybreaks

\theoremstyle{plain}
\newtheorem{theorem}{Theorem}[section]

\newtheorem{lemma}[theorem]{Lemma}
\newtheorem{corollary}[theorem]{Corollary}

\theoremstyle{definition}

\newtheorem{example}[theorem]{Example}

\theoremstyle{remark}
\newtheorem{remark}[theorem]{Remark}

\begin{document}

\title[Almost Ricci-Yamabe solitons on Almost Kenmotsu Manifolds]
{Almost Ricci-Yamabe solitons on Almost Kenmotsu Manifolds}

\author[M. Khatri]{Mohan Khatri}
\address{Mohan Khatri \\ Department of Mathematics and Computer Science \\ Mizoram University \\ Aizawl-796004, India}
\email{mohankhatri.official@gmail.com}

\author[J.P. Singh]{Jay Prakash Singh}
\address{Jay Prakash Singh \\ Department of Mathematics and Computer Science \\ Mizoram University \\ Aizawl-796004, India}
\email{jpsmaths@gmail.com}


\begin{abstract}
This manuscript examines almost Kenmotsu manifolds (briefly, AKMs) endowed with the almost Ricci-Yamabe solitons (ARYSs) and gradient ARYSs. The condition for an AKM with ARYS to be $\eta$-Einstein is established. We also show that an ARYS on Kenmotsu manifold becomes Ricci-Yamabe soliton under certain restrictions. In this series, it is proven that a $(2n+1)$-dimensional $(\kappa, \mu)'$-AKM equipped with a gradient ARYS is either locally isometric to $\mathbb{H}^{n+1}(-4)\times\mathbb{R}^n$ or the Reeb vector field and the soliton vector field are codirectional. The properties of $3$-dimensional non-Kenmotsu AKMs endowed with a gradient ARYS are studied.
\end{abstract}

\maketitle

\section{Introduction}
The Ricci flow is a well-known geometric flow, introduced by Hamilton \cite{1} in 80's, and it is used to prove the three-dimensional sphere theorem \cite{2}. The Ricci flow plays a crucial role in the proof of Poincar\'{e} and Thurston's conjectures. The Ricci soliton is a special self-similar solution of the Hamilton's Ricci flow: $\frac{\partial}{\partial t}g(t)=-2S(t),$ with initial condition $g(0)=g_0$. A Ricci soliton $(g, V, \lambda)$ on an $m$-dimensional Riemannian manifold $(M,g)$ is defined by
\begin{eqnarray}
(\mathcal{L}_Vg)(\Omega_{1},\Omega_{2})+2S(\Omega_{1},\Omega_{2})=2\lambda g(\Omega_{1},\Omega_{2}),\nonumber
\end{eqnarray}
for any vector fields $\Omega_{1},\Omega_{2}$ on $M$. Here, $\mathcal{L}_Vg$ denotes the Lie-derivative of $g$ along a potential vector field $V$, $S$ is the Ricci tensor of $M$ and $\lambda$, real constant (soliton constant). It is said to be shrinking, steady and expanding accordingly as $\lambda$ is positive, zero and negative, respectively. When the vector field $V$ is zero, it is said to be trivial and when $V$ is the gradient of a smooth function $f$ on $M$, that is, $V=\nabla f$, where $\nabla$ is the covarient operator of $g$,  then we say that the Ricci soliton reduces to the gradient Ricci soliton. For details see \cite{chaubey1,3,Pe}.\\

Similar to Ricci flow, Hamilton introduced Yamabe flow to tackle the Yamabe's problem on manifolds of positive conformal Yamabe invariant. The Yamabe soliton is a self-similar solution to the Yamabe flow. On a Riemannian manifold $(M,g)$, a Yamabe soliton is given by
\begin{eqnarray}
(\mathcal{L}_Vg)(\Omega_{1},\Omega_{2})=2(r-\lambda)g(\Omega_{1},\Omega_{2})\nonumber
\end{eqnarray}
for any vector fields $\Omega_{1},\Omega_{2}$ on $M$. Here, $r$ is the scalar curvature of the manifold and $\lambda$, real constant. The Yamabe soliton preserves the conformal class of the metric but the Ricci soliton does not in general. Moreover, in dimension $n=2$, both the solitons are similar. If $\lambda$ is a smooth function on $M$, then the Ricci solitons and the Yamabe solitons are called almost Ricci solitons \cite{4} and almost Yamabe solitons \cite{5}, respectively.\\

In \cite{9}, G\"{u}ler and Crasmareanu gave  the notion of Ricci-Yamabe flow.  
A Ricci-Yamabe soliton (RYS) of type $(\alpha, \beta)$ 
 on $M$ emerges as the limit of the solution of Ricci-Yamabe flow, and it is goverened by the following equation: 
\begin{eqnarray}\label{b1}
\mathcal{L}_Vg+2\alpha S=(2\lambda -\beta r)g,
\end{eqnarray}
where $\lambda,\alpha,\beta \in \mathbb{R}$. If the soliton vector $V$ is a gradient of some smooth function $f$ on $M$, then (\ref{b1}) becomes gradient RYS equation, that is, 
\begin{eqnarray}\label{b2}
\nabla^2f+\alpha S=(\lambda-\frac{1}{2}\beta r)g,
\end{eqnarray}
where $\nabla^2f$ is the Hessian of $f$.
An RYS is said to be expanding, shrinking or steady if $\lambda$ is negative, positive or zero, respectively. 
Particularly, $(1,0)$, $(0,1)$, $(1,-1)$ and $(1,-2\rho)$-type RYSs are Ricci solitons, Yamabe solitons, Einstein solitons and $\rho$-Einstein solitons, respectively. This infers that the notion of RYS is a natural generalization of a large class of soliton-like equations. If we replace $\lambda$ in equation (\ref{b1}) by a smooth function, then we say that $(M,g)$ is an ARYS.

Recently, in \cite{10}, the author studied RYSs on almost Kenmotsu manifolds. He showed that a $(k,\mu)'$-AKM admitting a RYS or gradient RYS is locally isometric to a product manifold. 
Siddiqi and Akyol \cite{11} have introduced the notion of $\eta$-RYS and established the geometrical bearing on Riemannian submersions in terms of $\eta$-RYS with the potential field. 
Yolda\c{s} \cite{Y} studied the geometrical properties of Kenmotsu manifolds with $\eta$-RYSs. In \cite{has1, has2}, authors established several results of RYSs within the framework of indefinite Kenmotsu manifolds and Riemannian manifolds, respectively. In \cite{J}, authors explored the applications of RYS in perfect fluid spacetimes. Singh and Khatri analyzed ARYS in almost contact metric manifolds and obtained several results. The last authors also analysed ARYS with certain vector fields \cite{M1}. In this sequel,  we study ARYSs and gradient ARYSs within the framework of  AKMs. \\

The structure of this manuscript is as follows:
In Section 2 is preliminaries. Normal AKMs endowed with ARYSs and gradient ARYSs are discussed in Section 3. It is shown that a Kenmotsu manifold with ARYS is $\eta$-Einstein. Also, we prove that an ARYS in Kenmotsu manifold reduces to an RYS, provided $Hess_\lambda(\zeta,\zeta)$ is invariant along $\zeta$. 
In Section 4,  we prove that if a $(\kappa,\mu)'$-AKM with $h'\neq0$ admits a gradient ARYS, then either $M$ is locally isometric to $\mathbb{H}^{n+1}(-4)\times\mathbb{R}^n$ or the potential vector field is pointwise collinear with the Reeb vector field. In Section 5, we investigate gradient ARYS in 3-dimensional non-Kenmotsu AKM.


\section{Preliminaries}
A smooth manifold $M^{2n+1}$ (dim$M=2n+1$) satisfying equations
\begin{eqnarray}\label{4}
\varphi^2\Omega_{1}=-\Omega_{1}+\eta(\Omega_{1})\zeta,~~~~~~\eta(\Omega_{1})=g(\Omega_{1},\zeta), \forall \,\, \Omega_{1} \in \mathfrak{X}(M),
\end{eqnarray}
for a $(1,1)$-tensor field $\varphi$, a unit vector field $\zeta$ (called the Reeb vector field) and a 1-form $\eta$, is termed as an almost contact manifold \cite{16}. Here $\mathfrak{X}(M)$ represents the collection of smooth vector fields of $M^{2n+1}$. 
If the Riemannian metric $g$ of $M^{2n+1}$ satisfies
\begin{equation}\label{5}
g(\varphi \Omega_{1},\varphi \Omega_{2})=g(\Omega_{1},\Omega_{2})-\eta(\Omega_{1})\eta(\Omega_{2}), \,\, g(\Omega_{1},\zeta)=\eta(\Omega_{1}), \forall \,\, \Omega_{1}, \Omega_{2} \in \mathfrak{X}(M),
\end{equation}
then $M^{2n+1}$ is said to be an almost contact metric manifold (ACMM)  and the structure $(\phi, \zeta, \eta, g)$ on $M^{2n+1}$ is called as an almost contact metric structure \cite{16}.

An ACMM with $d\eta=0$ and $\Phi=2\eta\wedge\Phi$ is defined as an AKM, 
 where the fundamental 2-form $\Phi$ of ACMM is defined by $\Phi(\Omega_{1},\Omega_{2})=g(\Omega_{1},\varphi \Omega_{2})$  and $d$ stands for exterior derivative \cite{17}. On the product $M^{2n+1}\times\mathbb{R}$ of an ACMM $M^{2n+1}$ and $\mathbb{R}$, there exists an almost complex structure $J$ defined by $J\big(\Omega_{1},\mathcal{F}\frac{d}{dt}\big)=\big(\varphi \Omega_{1}-\mathcal{F}\zeta,\eta(\Omega_{1})\frac{d}{dt}\big),$ where $\Omega_{1}$ denotes a vector field tangent to $M^{2n+1}$, $t$ is the coordinate of $\mathbb{R}$ and $\mathcal{F}$ is a $C^\infty$-function on $M^{2n+1}\times\mathbb{R}$. If $J$ is integrable, then $(\phi, \zeta, \eta, g)$ on $M^{2n+1}$ is said to be normal. A normal ACMM is called a Kenmotsu manifold \cite{18}. An ACMM is a Kenmotsu manifold if and only if 
\begin{equation}
(\nabla_{\Omega_{1}}\varphi)\Omega_{2}=g(\varphi \Omega_{1},\Omega_{2})\zeta-\eta(\Omega_{2})\varphi \Omega_{1},\nonumber
\end{equation}
which infers that
\begin{eqnarray}\label{6}
\nabla_{\Omega_{1}}\zeta=\Omega_{1}-\eta(\Omega_{1})\zeta,
\end{eqnarray}
\begin{eqnarray}\label{7}
R(\Omega_{1},\Omega_{2})\zeta=\eta(\Omega_{1})\Omega_{2}-\eta(\Omega_{2})\Omega_{1},
\end{eqnarray}
\begin{eqnarray}\label{8}
Q\zeta=-2n\zeta,
\end{eqnarray}
 Here $R$ is the curvature tensor of $g$ and $Q$ the Ricci operator associated with the Ricci tensor $S$ as $S(\Omega_{1},\Omega_{2})=g(Q\Omega_{1},\Omega_{2})$. It is shown that a Kenmotsu manifold is locally a warped product $I\times_fN^{2n}$, where $I$ is an open interval with coordinate $t$, $f=ce^t$ is the warping function for some positive constant $c$ and $N^{2n}$ is a K$\ddot{a}$hlerian manifold \cite{18}.

On an AKM the following formula is valid \cite{19,20}:
\begin{eqnarray}\label{9}
\nabla_{\Omega_{1}}\zeta=-\varphi^2\Omega_{1}-\varphi h\Omega_{1}, , \forall \,\, \Omega_{1} \in \mathfrak{X}(M).
\end{eqnarray}
Let define operators $h$ and $\ell$ as: $$h=\frac{1}{2}\mathcal{L}_\zeta\varphi,\,\, \ell=R(\cdot,\zeta)\zeta.$$ Then we have $h\zeta=h'\zeta=0$, $Tr (h)=Tr (h')=0,~~~h\varphi=-\varphi h$, where $h'=h\cdot\varphi$ and $Tr$ denotes trace.\\

\section{Normal almost Kenmotsu manifold}
In this section, we deal with normal AKM, that is, Kenmotsu manifold admitting ARYS and gradient ARYS. Firstly, we give some examples of gradient ARYS.
\begin{example}
Let $(N,J,g_0)$ be a K$\ddot{a}$hler manifold of dimension $2n$. Consider the warped product $(M,g)=(\mathbb{R}\times_\sigma N, dt^2+\sigma^2g_0)$, where $t$ is the coordinate on $\mathbb{R}$. We set $\eta=dt$, $\zeta=\frac{\partial}{\partial t}$ and (1,1) tensor field $\varphi$ by $\varphi \Omega_{1}=J\Omega_{1}$ for vector field $\Omega_{1}$ on $N$ and $\varphi \Omega_{1}=0$ if $\Omega_{1}$ is tangent to $\mathbb{R}$. The above warped product with the structure $(\varphi,\zeta,\eta,g)$ is a Kenmotsu manifold \cite{18}. In particular, if we take $N=\mathbb{CH}^{2n}$, then $N$ being Einstein, the Ricci tensor of $M$ becomes $S^M=-2ng$. Then it is easy to verify that $(M,f,g,\lambda)$ is an ARYS for $f(x,t)=ke^t, k>0$ and $\lambda(x,t)=-2n\alpha-n\beta(2n+1)+ke^t$.
\end{example}
Therefore, a large number of examples can be constructed by considering different potential functions $f$ on warped product spaces. Next, we constructed an example by using Kanai's result \cite{K}.
\begin{example}
Let $N^{2n}$ be a complete Einstein K$\ddot{a}$hler manifold with $S^N=-(2n-1)g_0$. Now consider the warped product $M^{2n+1}=\mathbb{R}\times_{cosht}N^{2n}$ with the metric $g=dt^2+(cosht)^2g_0$. Then by using result by Kanai \cite{K}, there exists a function $f$ on $M$ without critical points satisfying $\nabla^2f=-fg$. Then it is easy to see that $(M,g,\nabla f,\lambda)$ is an ARYS for $\lambda=-2n\alpha-f-n\beta(2n+1)$.
\end{example}
Ghosh \cite{G1} initiated the study of Ricci soliton in Kenmotsu 3-manifold. He later studied gradient ARYS in Kenmotsu manifold and obtained Theorem 3 (see \cite{Gho}). Here, we generalized these results for ARYS and prove.
\begin{theorem}\label{t1}
If the metric of a Kenmotsu manifold $M^{2n+1}(\varphi,\zeta,\eta,g)$ admits a gradient ARYS with $\alpha\neq0$, then it is $\eta$-Einstein. Moreover, if $\zeta$ leaves the scalar curvature invariant then $\lambda$ can be expressed locally as $\lambda=Acosht+Bsinht-2n\alpha-n\beta(2n+1)$, where $A,B$ are constants on $M$.
\end{theorem}
\begin{proof}
Suppose the metric $g$ of Kenmtosu manifold admits gradient RYS, then from (\ref{b2}) we have
\begin{eqnarray}\label{a1}
\nabla_{\Omega_{1}}Df=\sigma \Omega_{1}-\alpha Q\Omega_{1},
\end{eqnarray}
for any vector field $\Omega_{1}$ on $M$ and $\sigma=\lambda-\frac{\beta r}{2}$ is a smooth function on $M$. \\
Taking an inner product of (\ref{a1}) along arbitrary vector field $\Omega_{2}$, we obtain:
\begin{eqnarray}\label{a2}
\nabla_{\Omega_{2}}\nabla_{\Omega_{1}}Df=(\Omega_{2}\sigma)\Omega_{1}+\sigma(\nabla_{\Omega_{2}}\Omega_{1})-\alpha(\nabla_{\Omega_{2}}Q)\Omega_{1}-\alpha Q(\nabla_{\Omega_{2}}\Omega_{1}).
\end{eqnarray}
Making use of (\ref{a2}) in the well-known formula $R(\Omega_{1},\Omega_{2})Df=\nabla_{\Omega_{1}}\nabla_{\Omega_{2}}Df-\nabla_{\Omega_{2}}\nabla_{\Omega_{1}}Df-\nabla_{[\Omega_{1},\Omega_{2}]}Df$ yields
\begin{eqnarray}\label{a3}
R(\Omega_{1},\Omega_{2})Df=(\Omega_{1}\sigma)\Omega_{2}-(\Omega_{2}\sigma)\Omega_{1}-\alpha[(\nabla_{\Omega_{1}}Q)\Omega_{2}-(\nabla_{\Omega_{2}}Q)\Omega_{1}].
\end{eqnarray}
Taking a covariant derivative of (\ref{8}) and using (\ref{6}), we get $(\nabla_{\Omega_{1}}Q)\zeta=-2n(\Omega_{1}-\eta(\Omega_{1})\zeta)$. Because of this in the inner product of (\ref{a3}) with $\zeta$ gives
\begin{eqnarray}\label{a4}
g(R(\Omega_{1},\Omega_{2})Df,\zeta)=(\Omega_{1}\sigma)\eta(\Omega_{2})-(\Omega_{2}\sigma)\eta(\Omega_{1}).
\end{eqnarray}
Now, taking an inner product of (\ref{7}) with $Df$ yields
\begin{eqnarray}\label{a5}
g(R(\Omega_{1},\Omega_{2})\zeta,Df)=(\Omega_{2}f)\eta(\Omega_{1})-(\Omega_{1}f)\eta(\Omega_{2}).
\end{eqnarray}
Combining (\ref{a4}) and (\ref{a5}) and replacing $\Omega_{2}$ by $\zeta$ in the obtain relations we obtain
\begin{eqnarray}\label{a6}
d(\sigma-f)=\zeta(\sigma-f)\eta,
\end{eqnarray}
where $d$ is the exterior derivative. This means that $\sigma-f$ is invariant along the distribution $\mathcal{D}$ (i.e., $\mathcal{D}=ker\eta)$ hence $\sigma-f$ is constant for all $\Omega_{1}\in\mathcal{D}$.\\
Contracting (\ref{a3}) infer
\begin{eqnarray}\label{a7}
S(\Omega_{2},Df)=-2n(\Omega_{2}\sigma)+\frac{\alpha}{2}(\Omega_{2}r),
\end{eqnarray}
for any vector field $\Omega_{2}$ on $M$. Replacing $\Omega_{2}$ by $\zeta$ in (\ref{a3}) and taking an inner product with $\Omega_{2}$ gives
\begin{eqnarray}\label{a8}
g(R(\Omega_{1},\zeta)Df,\Omega_{2})=(\Omega_{1}\sigma)\eta(\Omega_{2})-(\zeta\sigma)g(\Omega_{1},\Omega_{2})-\alpha S(\Omega_{1},\Omega_{2})+2n\alpha g(\Omega_{1},\Omega_{2}).
\end{eqnarray}
In consequence of (\ref{6}) and (\ref{7}) in (\ref{a8}) we get
\begin{eqnarray}\label{a9}
[(\Omega_{1}f)-(\Omega_{1}\sigma)]\eta(\Omega_{2})+\zeta(\sigma-f)g(\Omega_{1},\Omega_{2})+\alpha S(\Omega_{1},\Omega_{2})+2n\alpha g(\Omega_{1},\Omega_{2})=0.
\end{eqnarray}
Contracting (\ref{a9}) over $\Omega_{1}$ gives
\begin{eqnarray}\label{a10}
2n\zeta(\sigma-f)+\alpha[r+2n(2n+1)]=0.
\end{eqnarray}
Replacing $\Omega_{2}$ by $\zeta$ in (\ref{a7}) and making use of (\ref{a10}) and (\ref{8}), we see that $\zeta r=-2(r+2n(2n+1))$, for $\alpha\neq0$. In consequence, (\ref{a10}) in (\ref{a6}) gives
\begin{eqnarray}\label{a11}
d(\sigma-f)=-\alpha(\frac{r}{2n}+2n+1)\eta.
\end{eqnarray}
Applying Poincare lemma and using the fact that $d\eta=0$ on (\ref{a11}), we obtain $-\alpha dr\wedge\eta=0$, and making use of the value of $\zeta r$ we have
\begin{eqnarray}\label{a12}
Dr=-2(r+2n(2n+1))\zeta.
\end{eqnarray}
Taking an inner product of (\ref{a11}) with vector field $\Omega_{1}$, then inserting it along with (\ref{a10}) in (\ref{a9}) we get
\begin{eqnarray}\label{a13}
Q\Omega_{1}=(\frac{r}{2n}+1)\Omega_{1}-(\frac{r}{2n}+2n+1)\eta(\Omega_{1})\zeta,
\end{eqnarray}
for any vector field $\Omega_{1}$ on $M$. Therefore, $M$ is $\eta$-Einstein.\\
\indent Suppose that $\zeta r=0$, i.e., $\zeta$ leaves the scalar curvature invariant. In consequence of this we get $r=-2n(2n+1)$, a constant. Inserting this in (\ref{a13}) implies $Q\Omega_{1}=-2n\Omega_{1}$ i.e., $M$ is Einstein. As $r$ is constant, (\ref{a11}) gives $Df=D\lambda$. In consequence of this, (\ref{a1}) becomes
\begin{eqnarray}\label{a15}
\nabla_{\Omega_{1}}D\lambda=(\lambda+k)\Omega_{1},
\end{eqnarray}
where $k=n(2\alpha+\beta(2n+1))$. Replacing $\Omega_{1}$ by $\zeta$ and taking inner product with $\zeta$, (\ref{a15}) gives $\zeta(\zeta\lambda)=\lambda+k$. But as we know that a Kenmotsu manifold is locally isometric to the warped product $(-\epsilon,\epsilon)\times_{ce^t}N$, where $N$ is a K$\ddot{a}$hler manifold of dimension $2n$ and $(-\epsilon,\epsilon)$ is an open interval. Using the local parametrization: $\zeta=\frac{\partial}{\partial t}$ (where $t$ is the coordinate on $(-\epsilon,\epsilon)$) we get from (\ref{a15})
$$\frac{\partial^2\lambda}{\partial t^2}=\lambda+2n\alpha+n\beta(2n+1)$$
Its solution can be exhibited as $\lambda=Acosht+Bsinht-2n\alpha-n\beta(2n+1)$, where $A,B$ are constants on $M$. This completes the proof.
\end{proof}
\begin{lemma}\label{l1}
If the metric of a Kenmotsu manifold $M^{2n+1}(\varphi,\zeta,\eta,g)(n>1)$ admits ARYS then \\
1. $\zeta(\zeta\lambda)+\zeta\lambda=2(2n\alpha+\lambda+n\beta(2n+1)).$\\
2. $D\lambda=(\zeta\lambda)\zeta+\beta\{(r+2n(2n+1))\zeta+\frac{Dr}{2}\}.$
\end{lemma}
\begin{proof}
Taking the covariant derivative of (\ref{b1}) along arbitrary vector field $\Omega_{1}$, we get
\begin{eqnarray}\label{a16}
(\nabla_{\Omega_{1}}\mathcal{L}_Vg)(\Omega_{2},\Omega_{3})=2(\Omega_{1}\sigma)g(\Omega_{2},\Omega_{3})-2\alpha(\nabla_{\Omega_{1}}S)(\Omega_{2},\Omega_{3}),
\end{eqnarray}
where $\sigma=\lambda-\frac{\beta r}{2}$. We know the following commutative formula (see \cite{Ya}):
\begin{eqnarray}\label{a17}
(\mathcal{L}_V\nabla_{\Omega_{1}}g-\nabla_{\Omega_{1}}\mathcal{L}_Vg-\nabla_{[V,\Omega_{1}]}g)(\Omega_{2},\Omega_{3})\nonumber\\=-g((\mathcal{L}_V\nabla)(\Omega_{1},\Omega_{2}),\Omega_{3})-g((\mathcal{L}_V\nabla)(\Omega_{1},\Omega_{3}),\Omega_{2}),
\end{eqnarray}
for all vector fields $\Omega_{1},\Omega_{2},\Omega_{3}$ on $M$. Since $g$ is parallel with respect to Levi-Civita connection $\nabla$, the above relation becomes:
\begin{eqnarray}\label{a18}
(\nabla_{\Omega_{1}}\mathcal{L}_Vg)(\Omega_{2},\Omega_{3})=g((\mathcal{L}_V\nabla)(\Omega_{1},\Omega_{2}),\Omega_{3})+g((\mathcal{L}_V\nabla)(\Omega_{1},\Omega_{3}),\Omega_{2}).
\end{eqnarray}
We know that $\mathcal{L}_V\nabla$ is a symmetric tensor of type (1,2) and so it follows from (\ref{a18}) that
\begin{eqnarray}\label{a19}
2g((\mathcal{L}_V\nabla)(\Omega_{1},\Omega_{2}),\Omega_{3})=(\nabla_{\Omega_{1}}\mathcal{L}_Vg)(\Omega_{2},\Omega_{3})+(\nabla_{\Omega_{2}}\mathcal{L}_Vg)(\Omega_{3},\Omega_{1})-(\nabla_{\Omega_{3}}\mathcal{L}_Vg)(\Omega_{1},\Omega_{2}).
\end{eqnarray}
Inserting (\ref{a16}) in (\ref{a19}), then replacing $\Omega_{2}$ by $\zeta$ we obtain
\begin{eqnarray}\label{a20}
(\mathcal{L}_V\nabla)(\Omega_{1},\zeta)=2\alpha Q\Omega_{1}+(4n\alpha+\zeta\sigma)\Omega_{1}+g(\Omega_{1},D\sigma)\zeta-\eta(\Omega_{1})D\sigma.
\end{eqnarray}
Taking the covariant derivative of (\ref{a20}) along arbitrary vector field $\Omega_{2}$ gives
\begin{eqnarray}\label{a21}
(\nabla_{\Omega_{2}}\mathcal{L}_V\nabla)(\Omega_{1},\zeta)+(\mathcal{L}_V\nabla)(\Omega_{1},\Omega_{2})-\eta(\Omega_{2})(\mathcal{L}_V\nabla)(\Omega_{1},\zeta)\nonumber\\=2\alpha(\nabla_{\Omega_{2}}Q)\Omega_{1}+\Omega_{2}(\zeta\sigma)\Omega_{1}+g(\Omega_{1},\nabla_{\Omega_{2}}D\sigma)\zeta\nonumber\\-g(\Omega_{1},D\sigma)\varphi^2{\Omega_{2}}+g(\Omega_{1},\varphi^2\Omega_{2})D\sigma-\eta(\Omega_{1})(\nabla_{\Omega_{2}}D\sigma).
\end{eqnarray}
Making use of this in the formula (see \cite{Ya}) $$(\mathcal{L}_VR)(\Omega_{1},\Omega_{2})\Omega_{3}=(\nabla_{\Omega_{1}}\mathcal{L}_V\nabla)(\Omega_{2},\Omega_{3})-(\nabla_{\Omega_{2}}\mathcal{L}_V\nabla)(\Omega_{1},\Omega_{3})$$
we obtain
\begin{eqnarray}\label{a22}
(\mathcal{L}_VR)(\Omega_{1},\Omega_{2})\zeta=2\alpha\{(\nabla_{\Omega_{1}}Q)\Omega_{2}-(\nabla_{\Omega_{2}}Q)\Omega_{1}\}+\Omega_{1}(\zeta\sigma)\Omega_{2}-\Omega_{2}(\zeta\sigma)\Omega_{1}\nonumber\\+g(\Omega_{2},D\sigma)\Omega_{1}-g(\Omega_{1},D\sigma)\Omega_{2}+\eta(\Omega_{1})\nabla_{\Omega_{2}}D\sigma-\eta(\Omega_{2})\nabla_{\Omega_{1}}D\sigma\nonumber\\+2\alpha\{\eta(\Omega_{1})Q\Omega_{2}-\eta(\Omega_{2})Q\Omega_{1}\}+(4\alpha n+\zeta\sigma)\{\eta(\Omega_{1})\Omega_{2}-\eta(\Omega_{2})\Omega_{1}\}.
\end{eqnarray}
Now differentiating $\zeta\sigma=g(\zeta,D\sigma)$ along vector field $\Omega_{1}$ and using (\ref{6}) we get
\begin{eqnarray}\label{a23}
\Omega_{1}(\zeta\sigma)=g(\Omega_{1},D\sigma)-(\zeta\sigma)\eta(\Omega_{1})+g(\nabla_{\Omega_{1}}D\sigma,\zeta).
\end{eqnarray}
Replacing $\Omega_{2}$ by $\zeta$ in (\ref{b1}), then inserting it in the Lie-derivative of (\ref{7}) yields
\begin{eqnarray}\label{a24}
(\mathcal{L}_VR)(\Omega_{1},\Omega_{2})\zeta+R(\Omega_{1},\Omega_{2})\mathcal{L}_V\zeta=g(\Omega_{1},\mathcal{L}_V\zeta)\Omega_{2}\nonumber\\-g(\Omega_{2},\mathcal{L}_V\zeta)\Omega_{1}+2(\sigma+2\alpha n)\{\eta(\Omega_{1})\Omega_{2}-\eta(\Omega_{2})\Omega_{1}\}.
\end{eqnarray}
Combining (\ref{a22}), (\ref{a23}) and (\ref{a24}), we obtain
\begin{eqnarray}\label{a25}
&g(\Omega_{1},\mathcal{L}_V\zeta)\Omega_{2}-g(\Omega_{2},\mathcal{L}_V\zeta)\Omega_{1}-R(\Omega_{1},\Omega_{2})\mathcal{L}_V\zeta\nonumber\\&=2\alpha\{(\nabla_{\Omega_{1}}Q)\Omega_{2}-(\nabla_{\Omega_{2}}Q)\Omega_{1}+\eta(\Omega_{1})Q\Omega_{2}\nonumber\\&-\eta(\Omega_{2})Q\Omega_{1}\}+g(\nabla_{\Omega_{1}}D\sigma,\zeta)\Omega_{2}-g(\nabla_{\Omega_{2}}D\sigma,\zeta)\Omega_{1}\nonumber\\&+\eta(\Omega_{1})\nabla_{\Omega_{2}}D\sigma-\eta(\Omega_{2})\nabla_{\Omega_{1}}D\sigma-2\sigma\{\eta(\Omega_{1})\Omega_{2}-\eta(\Omega_{2})\Omega_{1}\}.
\end{eqnarray}
Replacing $\Omega_{1}$ and $\Omega_{2}$ by $\varphi \Omega_{1}$ and $\varphi \Omega_{2}$ in (\ref{a25}) then contracting the obtain equation and using Lemma 4.2 (see \cite{Gh}) results in
\begin{eqnarray}\label{a26}
S(\Omega_{2},\mathcal{L}_V\zeta)+2ng(\Omega_{2},\mathcal{L}_V\zeta)=\alpha(\Omega_{2}r)+2\alpha(r+4n^2+2n)\eta(\Omega_{2})-g(\nabla_\zeta D\lambda,\varphi^2\Omega_{2}).\nonumber
\end{eqnarray}
Contracting (\ref{a25}) and combining it with the forgoing equation yields
\begin{eqnarray}\label{a27}
2(n-1)g(\nabla_\zeta D\sigma,\Omega_{2})+\zeta(\zeta\sigma)\eta(\Omega_{2})+\eta(\Omega_{2})divD\sigma=4n(2n\alpha+\sigma)\eta(\Omega_{2}).
\end{eqnarray}
Replacing $\Omega_{2}$ by $\zeta$ in (\ref{a27}), we get $(2n-1)\zeta(\zeta\sigma)+divD\sigma=4n(2n\alpha+\sigma)$. In view of this in (\ref{a27}) infer $g(\nabla_\zeta D\sigma,\Omega_{1})=\zeta(\zeta\sigma)\eta(\Omega_{1})$ for $n>1$. Now taking $\zeta$ instead of $\Omega_{2}$ in (\ref{a25}) and making use of above relations we obtain
\begin{eqnarray}\label{a28}
\nabla_{\Omega_{1}}D\sigma=-2(2n\alpha+\sigma)\varphi^2\Omega_{1}+\zeta(\zeta\sigma)\varphi^2\Omega_{1}+\zeta(\zeta\sigma)\eta(\Omega_{1})\zeta.
\end{eqnarray}
In consequence of (\ref{a28}), the expression of the curvature tensor is as follows:
\begin{eqnarray}\label{a29}
R(\Omega_{1},\Omega_{2})D\sigma=2(\Omega_{2}\sigma)\varphi^2\Omega_{1}-2(\Omega_{1}\sigma)\varphi^2\Omega_{2}+\Omega_{2}(\zeta(\zeta\sigma))\Omega_{1}\nonumber\\-\Omega_{1}(\zeta(\zeta\sigma))\Omega_{2}+2(\zeta(\zeta\sigma)-\sigma-2n\alpha)\{\eta(\Omega_{2})\Omega_{1}-\eta(\Omega_{1})\Omega_{2}\}.
\end{eqnarray}
Replacing $\Omega_{2}$ by $\zeta$ in (\ref{a29}), then inserting the obtain equation back in (\ref{a29}) gives
\begin{eqnarray}\label{a30}
R(\Omega_{1},\Omega_{2})D\sigma=(\Omega_{1}\sigma)\Omega_{2}-(\Omega_{2}\sigma)\Omega_{1}-2\{(\Omega_{1}\sigma)\eta(\Omega_{2})\zeta\nonumber\\-(\Omega_{2}\sigma)\eta(\Omega_{1})\zeta\}+\{\zeta(\zeta(\zeta\sigma))-\zeta\sigma\}\{\eta(\Omega_{2})\Omega_{1}-\eta(\Omega_{1})\Omega_{2}\}\nonumber\\+2\{\zeta(\zeta\sigma)-\sigma-2n\alpha\}\{\eta(\Omega_{2})\Omega_{1}-\eta(\Omega_{1})\Omega_{2}\}.
\end{eqnarray}
Replacing $\Omega_{1}$ and $\Omega_{2}$ by $\varphi \Omega_{1}$ and $\varphi \Omega_{2}$ in (\ref{a30}) then contracting the obtained result yields
\begin{eqnarray}
S(\Omega_{2},D\sigma)=-2ng(\Omega_{2},D\sigma).\nonumber
\end{eqnarray}
In consequence of this in the contraction of (\ref{a30}) and further replacing $\Omega_{2}$ by $\varphi \Omega_{2}$ in the obtained expression yields $\varphi D\sigma=0$. Differentiating this along vector field $\Omega_{1}$ and inserting it in (\ref{a28}) along with the fact that $\sigma=\lambda-\frac{\beta r}{2}$ and (\ref{6}) gives
\begin{eqnarray}
\zeta(\zeta\lambda)+\zeta\lambda=2(2n\alpha+\lambda+n\beta(2n+1)).
\end{eqnarray}
This completes the proof.
\end{proof}
\begin{theorem}\label{t2}
Let $M^{2n+1}(\varphi,\zeta,\eta,g)(n>1)$ be a Kenmotsu manifold whose metric represents an ARYS. If $Hess_\lambda(\zeta,\zeta)$ is constant along Reeb vector field then it reduces to RYS with $\lambda=-2n\alpha-n\beta(2n+1)$.
\end{theorem}
\begin{proof}
By hypothesis, $Hess_\lambda(\zeta,\zeta)$ is constant along the Reeb vector field $\zeta$ i.e., $\zeta(\zeta(\zeta\lambda))=0$ implies $\zeta(\zeta\lambda)$ is constant along $\zeta$. In view of this in the covariant derivative of first relation in Lemma (\ref{l1}) along $\zeta$, we get $\zeta(\zeta\lambda)=2(\zeta\lambda)$. Again differentiating this along $\zeta$ yields $\zeta\lambda=0$, that is, $\lambda$ is constant along $\zeta$. Making use of this in first relation of Lemma (\ref{l1}) gives $\lambda=-2n\alpha-n\beta(2n+1)$. Therefore ARYS reduces to RYS. This completes the proof.
\end{proof}
\begin{remark}
The above Theorem \ref{t2} is a generalization of Theorem 4.1 in Ghosh \cite{Gh}, where he obtained the condition under which a Ricci almost soliton reduces to an expanding Ricci soliton with $\lambda=-2n$. It is easy to see that for $\alpha=1$ and $\beta=0$ Theorem 4.1 \cite{Gh} can be obtained from Theorem \ref{t2}. Moreover, the first condition of Theorem 4.1 \cite{Gh} is also true for this choice of scalars.
\end{remark}
A Kenmotsu manifold is said to be $\eta$-Einstein if there exists smooth functions $a$ and $b$ such that
\begin{eqnarray}\label{c1}
S(\Omega_{1},\Omega_{2})=ag(\Omega_{1},\Omega_{2})+b\eta(\Omega_{1})\eta(\Omega_{2}),
\end{eqnarray}
for all vector field $\Omega_{1},\Omega_{2}$ on $M$. If $b=0$, then $M$ becomes an Einstein manifold.

\begin{theorem}\label{t3}
If the metric of an $\eta$-Einstein Kenmotsu manifold $M^{2n+1}(\varphi,\zeta,\eta,g)(n>1)$ admits RYS with $\alpha\neq0$ then it is Einstein with constant scalar curvature $r=-2n(2n+1)$, provided $2\alpha+n\beta\neq0$.
\end{theorem}
\begin{proof}
Replacing $\Omega_{2}$ by $\zeta$ in (\ref{c1}) and using (\ref{8}), we get $a+b=-2n$. Then contracting (\ref{c1}) gives $r=(2n+1)a+b$. In view of this (\ref{c1}) becomes
\begin{eqnarray}\label{c2}
S(\Omega_{1},\Omega_{2})=(\frac{r}{2n}+1)g(\Omega_{1},\Omega_{2})-(\frac{r}{2n}+2n+1)\eta(\Omega_{1})\eta(\Omega_{2}),
\end{eqnarray}
for any vector field $\Omega_{1},\Omega_{2}$ on $M$. Making use of (\ref{c2}) in (\ref{b1}) yields
\begin{eqnarray}\label{c3}
(\mathcal{L}_Vg)(\Omega_{2},\Omega_{3})=\{2\lambda-\beta r-2\alpha(\frac{r}{2n}+1)\}g(\Omega_{2},\Omega_{3})+2\alpha\{(2n+1)+\frac{r}{2n}\}\eta(\Omega_{2})\eta(\Omega_{3}).
\end{eqnarray}
Taking the covariant derivative of (\ref{c3}) along arbitrary vector field $\Omega_{1}$ we obtain
\begin{eqnarray}\label{c4}
(\nabla_{\Omega_{1}}\mathcal{L}_Vg)(\Omega_{2},\Omega_{3})=-(\frac{\alpha}{n}+\beta)(\Omega_{1}r)g(\Omega_{2},\Omega_{3})+\frac{\alpha}{n}(\Omega_{1}r)\eta(\Omega_{2})\eta(\Omega_{3})\nonumber\\+2\alpha(\frac{r}{2n}+2n+1)\{g(\Omega_{1},\Omega_{2})\eta(\Omega_{3})+g(\Omega_{1},\Omega_{3})\eta(\Omega_{2})-2\eta(\Omega_{1})\eta(\Omega_{2})\eta(\Omega_{3})\}.
\end{eqnarray}
Making use of (\ref{c4}) in (\ref{a19}) yields
\begin{eqnarray}\label{c5}
2(\mathcal{L}_V\nabla)(\Omega_{1},\Omega_{2})=-(\frac{\alpha}{n}+\beta)\{(\Omega_{1}r)\Omega_{2}+(\Omega_{2}r)\Omega_{1}-g(\Omega_{1},\Omega_{2})Dr\}\nonumber\\+\frac{\alpha}{n}\{(\Omega_{1}r)\eta(\Omega_{2})\zeta+(\Omega_{2}r)\eta(\Omega_{1})\zeta-\eta(\Omega_{1})\eta(\Omega_{2})Dr\}\nonumber\\+4\alpha(2n+1+\frac{r}{2n})\{g(\Omega_{1},\Omega_{2})\zeta-\eta(\Omega_{1})\eta(\Omega_{2})\zeta\}.
\end{eqnarray}
Setting $\Omega_{1}=\Omega_{2}=e_i$ where $e_i:i=1,2,..,2n+1$ is an orthonormal frame in (\ref{c5}) and summing over $i$, we get
\begin{eqnarray}\label{c6}
2\sum_{i=1}^{2n+1}\varepsilon_i(\mathcal{L}_V\nabla)(e_i,e_i)=-\{\beta(1-2n)+\frac{\alpha}{n}(1-n)\}Dr\nonumber\\+\frac{2\alpha}{n}(\zeta r)\zeta+8n\alpha(\frac{r}{2n}+2n+1)\zeta.
\end{eqnarray}
Now taking the covariant derivative of (\ref{c3}) give
\begin{eqnarray}\label{c7}
(\nabla_{\Omega_{1}}\mathcal{L}_Vg)(\Omega_{2},\Omega_{3})=-\beta(\Omega_{1}r)g(\Omega_{2},\Omega_{3})-2\alpha(\nabla_{\Omega_{1}}S)(\Omega_{2},\Omega_{3}),\nonumber
\end{eqnarray}
which on contracting yields $2\sum_{i=1}^{2n+1}\varepsilon_i(\mathcal{L}_V\nabla)(e_i,e_i)=\beta(2n-1)Dr$. In consequence of this in (\ref{c6}), we obtain
\begin{eqnarray}\label{c8}
2\alpha(n-1)Dr+2\alpha(\zeta r)\zeta+8n\alpha(\frac{r}{2n}+2n+1)\zeta=0.
\end{eqnarray}
Taking an inner product of (\ref{c8}) with $\zeta$, we get $\zeta r=-2(r+2n(2n+1))$ for $\alpha\neq 0$. In view of this, (\ref{c8}) yields $Dr=(\zeta r)\zeta$ for $n>1$. Then, replacing $\Omega_{2}$ by $\zeta$ in (\ref{c5}) results in
\begin{eqnarray}\label{c9}
2(\mathcal{L}_V\nabla)(\Omega_{1},\zeta)=-\beta(\Omega_{1}r)\zeta+(\frac{\alpha}{n}+\beta)(\zeta r)\varphi^2\Omega_{1}.
\end{eqnarray}
Taking the covariant derivate of (\ref{c9}), then inserting it in Yano's result \cite{Ya}: $$(\mathcal{L}_VR)(\Omega_{1},\Omega_{2})\Omega_{3}=(\nabla_{\Omega_{1}}\mathcal{L}_V\nabla)(\Omega_{2},\Omega_{3})-(\nabla_{\Omega_{2}}\mathcal{L}_V\nabla)(\Omega_{1},\Omega_{3})$$, for $\Omega_{3}=\zeta$ yields
\begin{eqnarray}\label{c10}
2(\mathcal{L}_VR)(\Omega_{1},\Omega_{2})\zeta=-\beta\{\Omega_{1}(\Omega_{2}r)\zeta-\Omega_{2}(\Omega_{1}r)\zeta\}-\beta\{(\Omega_{1}r)\varphi^2\Omega_{2}-(\Omega_{2}r)\varphi^2\Omega_{1}\}\nonumber\\+(\frac{\alpha}{n}+\beta)\{\Omega_{1}(\zeta r)\varphi^2\Omega_{2}-\Omega_{2}(\zeta r)\varphi^2\Omega_{1}+2(\zeta r)\{\eta(\Omega_{2})\Omega_{1}-\eta(\Omega_{1})\Omega_{2}\}\}.
\end{eqnarray}
Contracting (\ref{c10}) over $\Omega_{1}$ gives $(\mathcal{L}_VS)(\Omega_{2},\zeta)=-n\beta(\zeta r)\eta(\Omega_{2})$. In consequence of this in the Lie-derivative of (\ref{8}), we obtain
\begin{eqnarray}\label{c11}
(\frac{r}{2n}+1)g(\Omega_{2},\mathcal{L}_V\zeta)-(\frac{r}{2n}+2n+1)\eta(\Omega_{2})\eta(\mathcal{L}_V\zeta)=\nonumber\\n\beta(\zeta r)\eta(\Omega_{2})-2n(2\lambda-\beta r+4\alpha n)\eta(\Omega_{2})-2ng(\Omega_{2},\mathcal{L}_V\zeta).
\end{eqnarray}
Replacing $\Omega_{2}$ by $\zeta$ in (\ref{c11}) then inserting back in (\ref{c11}) gives $\lambda=-2\alpha n-n\beta(2n+1)$. In view of this in (\ref{c3}) we get $\eta(\mathcal{L}_V\zeta)=-\frac{\beta}{2}(r+2n(2n+1))$. In consequence of this and $Dr=(\zeta r)\zeta$ in the Lie-derivative of $S(\Omega_{1},\zeta)=-2n\eta(\Omega_{1})$ we get
\begin{eqnarray}\label{c12}
(r+2n(2n+1))\{2\mathcal{L}_V\zeta-\frac{\beta}{2}(\zeta r)\zeta\}=0.
\end{eqnarray}
Thus we get either $r=-2n(2n+1)$ in this case $M$ is Einstein or $\mathcal{L}_V\zeta=\frac{\beta}{4}(\zeta r)\zeta$. Suppose $r\neq-2n(2n+1)$ in some open set $\mathcal{O}$ on $M$. Then using (\ref{6}) and (\ref{c12}) implies
\begin{eqnarray}\label{c13}
\nabla_\zeta V=V-\eta(V)\zeta-\frac{\beta}{4}(\zeta r)\zeta.
\end{eqnarray}
Taking $\Omega_{2}=\zeta$ in the commutative formula $(\mathcal{L}_V\nabla)(\Omega_{1},\Omega_{2})=\nabla_{\Omega_{1}}\nabla_{\Omega_{2}}V-\nabla_{\nabla_{\Omega_{1}}\Omega_{2}}V+R(V,\Omega_{1})\Omega_{2}$ and using (\ref{c13}), (\ref{c12}) and (\ref{c9}), we obtain
\begin{eqnarray}
(2\alpha+n\beta)(\zeta r)\varphi^2\Omega_{1}=0,\nonumber
\end{eqnarray}
for any vector field $\Omega_{1}$ on $\mathcal{O}$. This shows that $\zeta r=0$, that is, $r=-2n(2n+1)$, a contradiction. This completes the proof.
\end{proof}
In particular, if we take scalar $\alpha=1$ and $\beta=0$ then in regard of Theorem \ref{t3}, we can state the following:
\begin{corollary}\cite{G2}
If the metric of an $\eta$-Einstein Kenmotsu manifold $M^{2n+1}(\varphi,\zeta,\eta,g), n>1$ is a Ricci soliton then it is Einstein and the soliton is expanding.
\end{corollary}
For the case $\alpha=1$ and $\beta=-2\rho$, we can state the following:
\begin{corollary}
If the metric of an $\eta$-Einstein Kenmotsu manifold $M^{2n+1}(\varphi,\zeta,\eta,g)(n>1)$ admits $\rho$-Einstein soliton then it is Einstein with constant scalar curvature $r=-2n(2n+1)$, provided $\rho\neq\frac{1}{n}$.
\end{corollary}
It is known that the warped product $\mathbb{R}\times_{ce^t}V(k)$ where $V(k)$ is a K$\ddot{a}$hler manifold of constant holomorphic sectional curvature of dimension $2n$ admits Kenmotsu structure (see \cite{18}). Moreover, its sectional tensor is given by \cite{O}
\begin{eqnarray}
R(\Omega_{1},\Omega_{2})\Omega_{3}=H\{g(\Omega_{2},\Omega_{3})\Omega_{1}-g(\Omega_{1},\Omega_{3})\Omega_{2}\}+(H+1)\{g(\Omega_{1},\Omega_{3})\eta(\Omega_{2})\zeta\nonumber\\-g(\Omega_{2},\Omega_{3})\eta(\Omega_{1})\zeta+\eta(\Omega_{1})\eta(\Omega_{3})\Omega_{2}-\eta(\Omega_{2})\eta(\Omega_{3})\Omega_{1}+g(\Omega_{1},\varphi \Omega_{3})\varphi \Omega_{2}\nonumber\\-g(\Omega_{2},\varphi \Omega_{3})\varphi \Omega_{1}+2g(\Omega_{1},\varphi \Omega_{2})\varphi \Omega_{3}\}.\nonumber
\end{eqnarray}
Contracting the above equation we see that $$S(\Omega_{1},\Omega_{2})=2\{(n-1)H-1\}g(\Omega_{1},\Omega_{2})-2(n-1)(H+1)\eta(\Omega_{1})\eta(\Omega_{2}),$$ that is, it is $\eta$-Einstein. Now making use of Theorem (\ref{t3}) we get $H=-1$. Hence, we can state the following:
\begin{corollary}
If the metric of the warped product $\mathbb{R}\times_{ce^t}V(k)$, $(n>1)$ is a RYS with $\alpha\neq0$ then it is of constant curvature -1, provided $2\alpha+n\beta\neq 0$.
\end{corollary}
\section{Non-Normal almost Kenmotsu manifolds}
In this section, we give some examples of AKMs admitting ARYS and also investigated it in $(\kappa,\mu)'$-AKM. An AKM $M^{2n+1}(\varphi,\zeta,\eta,g)$ is said to be a $(\kappa,\mu)'$-AKM if $\zeta$ belongs to the $(\kappa,\mu)$-nullity distribution, i.e.,
\begin{eqnarray}\label{28}
R(\Omega_{1},\Omega_{2})\zeta=\kappa[\eta(\Omega_{2})\Omega_{1}-\eta(\Omega_{1})\Omega_{2}]+\mu[\eta(\Omega_{2})h'\Omega_{1}-\eta(\Omega_{1})h'\Omega_{2}],
\end{eqnarray}
for all vector fields $\Omega_{1},\Omega_{2}$ on $M$, where $\kappa,\mu$ are constants. Moreover if both $\kappa$ and $\mu$ are smooth functions in (\ref{28}), then $M$ is called a generalized $(\kappa,\mu)'$-AKM (see \cite{20,25,26}). On generalized $(\kappa,\mu)$ or $(\kappa,\mu)'$-AKM with $h\neq0$ (equivalently, $h'\neq0$), the following relations hold \cite{20}:
\begin{eqnarray}\label{30}
h'^2=(\kappa+1)\varphi^2,~~~~~h^2=(\kappa+1)\varphi^2,
\end{eqnarray}
\begin{eqnarray}\label{31}
Q\zeta=2n\kappa\zeta.
\end{eqnarray}
It follows from (\ref{30}) that $\kappa\leq -1$ and $\nu=\pm\sqrt{-\kappa-1}$, where $\nu$ is an eigenvalue corresponding to eigenvector $\Omega_{1}\in\mathcal{D}~~ (\mathcal{D}=Ker(\eta))$ of $h'$. The equality holds if and only if $h=0$ (equivalently, $h'=0$). Thus $h'\neq0$ if and only if $\kappa<-1$. First, we give some examples of AKM admitting ARYS.
\begin{example}
Let $(N,J,\bar{g})$ be a strictly almost K$\ddot{a}$hler Einstein manifold. We take $\eta=dt$, $\zeta=\frac{\partial}{\partial t}$ and (1,1)-tensor $\varphi$ by $\varphi \Omega_{1}=J\Omega_{1}$ for vector field $\Omega_{1}$ on $N$ and $\varphi \Omega_{1}=0$ if $\Omega_{1}$ is tangent to $\mathbb{R}$. Then it is known that $(M,g)=(\mathbb{R}\times_{ce^t}N, g_0+ce^{2t}\bar{g})$ together with the structure $(\varphi,\zeta,\eta,g)$ is an AKM (see \cite{18}). Also since $N$ is Einstein, we see $S^M=-2ng$. We define a smooth function $f(x,t)=t^2$ then it is easy to see that $(M,g,f,\lambda)$ is an ARYS for $\lambda(x,t)=-2n\alpha-n\beta(2n+1)+2$.
\end{example}
We can also construct an example of ARYS in AKM constructed by Barbosa-Ribeiro \cite{5}. 
\begin{example}
On the warped product $M=\mathbb{R}\times_{\sigma(t)}\mathbb{H}^{2n}$ consider the metric $g=dt^2+\sigma^2(t)g_0$, where $g_0$ is the standard metric on the hyperbolic space $\mathbb{H}^{2n}$. Then by Algere et al. \cite{A} result, it is easy to see that it is AKM. Let $\sigma(t)=cosht$ and $f(x,t)=sinht$ then $(M,g,f,\lambda)$ is an ARYS with $\lambda=sinht-n\beta(2n+1)-2n\alpha$.
\end{example}
Next, we restate some of the results obtained by Wang and Liu \cite{26} which will be used later in the prove of main theorems.
\begin{lemma}\cite{26}\label{l2}
Let $M^{2n+1}(\varphi,\zeta,\eta,g)$ be a generalized $(\kappa,\mu)'$-AKM with $h'\neq0$. For $n>1$, the Ricci operator $Q$ of $M$ can be expressed as 
\begin{eqnarray}
Q\Omega_{1}=-2n\Omega_{1}+2n(\kappa+1)\eta(\Omega_{1})\zeta-[\mu-2(n-1)h']\Omega_{1},\nonumber
\end{eqnarray}
for any vector field $\Omega_{1}$ on $M$. Further, if $\kappa$ and $\mu$ are constants and $n\geq1$, then $\mu=-2$ and hence
\begin{eqnarray}
Q\Omega_{1}=-2n\Omega_{1}+2n(\kappa+1)\eta(\Omega_{1})\zeta-2nh'\Omega_{1},
\end{eqnarray}
for any vector field $\Omega_{1}$ on $M$. In both cases, the scalar curvature of $M$ is $2n(\kappa-2n)$.
\end{lemma}

Here, we consider gradient ARYS in the context of  $(\kappa,\mu)'$-AKM and generalized Theorem 3.5 \cite{10} and Theorem 3.1 \cite{W}. We state and prove the following:
\begin{theorem}\label{t4}
If $M^{2n+1}(\varphi,\zeta,\eta,g)$ be a $(\kappa,\mu)'$-AKM with $h'\neq0$ admitting gradient ARYS then either $M$ is locally isometric to $\mathbb{H}^{2n+1}(-4)\times\mathbb{R}^n$ or potential vector field is pointwise collinear with the Reeb vector field.
\end{theorem}
\begin{proof}
Suppose that $(\kappa,\mu)'$-AKM admits gradient ARYS then Eq. (\ref{a1})-(\ref{a3}) is valid. Taking an inner product of (\ref{a3}) with $\zeta$ and inserting Lemma \ref{l2}, we obtain
\begin{eqnarray}\label{d1}
g(R(\Omega_{1},\Omega_{2})Df,\zeta)=(\Omega_{1}\lambda)\eta(\Omega_{2})-(\Omega_{2}\lambda)\eta(\Omega_{1})-\alpha\{g(Qh'\Omega_{2},\Omega_{1})-g(Qh'\Omega_{1},\Omega_{2})\}.
\end{eqnarray}
Taking an inner product of (\ref{28}) with $Df$, then inserting it in (\ref{d1}) and replacing $\Omega_{1}$ by $\zeta$ gives
\begin{eqnarray}\label{d2}
-(\zeta\lambda)\zeta+D\lambda=\kappa\{(\zeta f)\zeta-Df\}-\mu h'Df.
\end{eqnarray}
Contracting (\ref{a3}) over $\Omega_{1}$, we get $QDf=-2nD\lambda$. In consequence of this in Lemma \ref{l2} gives
\begin{eqnarray}\label{d3}
D\lambda-Df+(\kappa+1)(\zeta f)\zeta=h'Df.
\end{eqnarray}
Combining (\ref{d2}) and (\ref{d3}), we obtain
\begin{eqnarray}\label{d4}
\kappa\{(\zeta f)\zeta-Df\}+D\lambda+(\zeta\lambda)\zeta-2Df+2(\kappa+1)(\zeta f)\zeta=0.
\end{eqnarray}
Operating the forgoing equation by $\varphi$ yields
\begin{eqnarray}\label{d5}
\varphi D\lambda-(\kappa+2)\varphi Df=0,\nonumber
\end{eqnarray}
implies, $$D\lambda-(\kappa+2)Df\in\mathbb{R}\zeta.$$ Therefore, we can write $D\lambda=(\kappa+2)Df+s\zeta$, where $s$ is a smooth function. In view of this in (\ref{d2}) infer
\begin{eqnarray}\label{d6}
2(\kappa+1)Df+(s-\zeta\lambda-\kappa\zeta f)\zeta=2h'Df.
\end{eqnarray}
Operating (\ref{d6}) by $h'$ then inserting the obtain expression in (\ref{d6}) gives
\begin{eqnarray}\label{d7}
(\kappa+2)\varphi Df=0.\nonumber
\end{eqnarray}
Thus we have either $\kappa=-2$ or $Df=(\zeta f)\zeta$. \\
Suppose $\kappa=-2$. Then without loss of generality, we may choose $\nu=1$. Then we have from Theorem 5.1 of \cite{25} we get
$$R({\Omega_{1}}_\nu,{\Omega_{2}}_\nu){\Omega_{3}}_\nu=-4[g({\Omega_{2}}_\nu,{\Omega_{3}}_\nu){\Omega_{1}}_\nu-g({\Omega_{1}}_\nu,{\Omega_{3}}_\nu){\Omega_{2}}_\nu],$$
$$R({\Omega_{1}}_{-\nu},{\Omega_{2}}_{-\nu}){\Omega_{3}}_{-\nu}=0,$$
for any ${\Omega_{1}}_\nu,{\Omega_{2}}_\nu,{\Omega_{3}}_\nu\in [\nu]'$ and ${\Omega_{1}}_{-\nu},{\Omega_{2}}_{-\nu},{\Omega_{3}}_{-\nu}\in[-\nu]'$. In consequence of this the Proposition 4.1 and Proposition 4.3 of \cite{25} along with $\nu=1$ shows that it is locally isometric to $\mathbb{H}^{2n+1}(-4)\times\mathbb{R}^n$. This completes the proof.
\end{proof}
Suppose $\kappa\neq-2$. Then in regard of Theorem \ref{t4} we have $V=Df=(\zeta f)\zeta$. Take $F=\zeta f$ and taking the covariant derivative of $V=F\zeta$ along arbitrary vector field $\Omega_{1}$ we get
\begin{eqnarray}
\nabla_{\Omega_{1}}V=(\Omega_{1}F)\zeta+F(-\varphi^2\Omega_{1}+h'\Omega_{1}).\nonumber
\end{eqnarray}
Making use of this in (\ref{b1}) yields
\begin{eqnarray}\label{d8}
(\Omega_{1}F)\eta(\Omega_{2})+(\Omega_{2}F)\eta(\Omega_{1})+2Fg(\Omega_{1},\Omega_{2})-2F\eta(\Omega_{1})\eta(\Omega_{2})\nonumber\\+2Fg(h'\Omega_{1},\Omega_{2})=(2\lambda-\beta r)g(\Omega_{1},\Omega_{2})-2\alpha S(\Omega_{1},\Omega_{2}).
\end{eqnarray}
Replacing $\Omega_{2}$ by $\zeta$ in (\ref{d8}) gives
\begin{eqnarray}\label{d9}
\Omega_{1}F=(2\lambda-\beta r-4n\alpha\kappa-\zeta F)\eta(\Omega_{1}),
\end{eqnarray}
for any vector field $\Omega_{1}$ on $M$. Contracting (\ref{d8}) then inserting it in (\ref{d9}) and replacing $\Omega_{1}$ by $\zeta$ in the obtained expression we obtain
\begin{eqnarray}\label{d10}
F=\lambda-\frac{\beta r}{2}+2n\alpha.
\end{eqnarray}
Inserting (\ref{d9}) in (\ref{d8}) and comparing it with Lemma \ref{l2} yields $(F-2n\alpha)(\kappa+1)\varphi^2\Omega_{1}=0$ for any $\Omega_{1}$ on $M$. As $\kappa<-1$, we see that $F=2n\alpha$, in view of this in (\ref{d10}) implies $\lambda=\frac{\beta r}{2}$ i.e., a constant. Therefore, $M$ reduces to gradient RYS. Hence using Corollary 3.7 of \cite{10} we can state the following:
\begin{corollary}
Let $M^{2n+1}(\varphi,\zeta,\eta,g)$ be a non-Kenmotsu $(\kappa,\mu)'$-AKM with $\kappa\neq-2$ admitting a gradient ARYS then\\
1. The potential vector field $V$ is a constant multiple of $\zeta$.\\
2. $V$ is a strict infinitisimal contact transformation.\\
3. $V$ leaves $h'$ invariant.
\end{corollary}
Consider a generalized $(\kappa,\mu)'$-AKM of dimension three with $\kappa<-1$. If we assume that $\kappa$ is invariant along $\zeta$, then from (Proposition 3.2, \cite{25}) we have $\zeta(\kappa)=-2(\kappa+1)(\mu+2)$ implies $\mu=-2$. Moreover, from (Lemma 3.3, \cite{Sa}), we have $h'(grad\mu)=grad\kappa-\zeta(\kappa)\zeta$ which implies $\kappa$ is constant under our assumption. Therefore $M^3$ becomes a $(\kappa,-2)'$-AKM. By applying Theorem \ref{t4}, we can conclude the following:
\begin{corollary}
Let $M^3(\varphi,\eta,\zeta,g)$ be a generalized non-Kenmotsu $(\kappa,\mu)'$-AKM with $\kappa<-1$ invariant along the Reeb vector field admitting gradient ARYS then it is either locally isometric to the product space $\mathbb{H}^{2n+1}(-4)\times\mathbb{R}^n$ or potential vector field is pointwise collinear with the Reeb vector field.
\end{corollary}
\section{3-dimensional Almost Kenmotsu manifolds}
On a 3-dimensional AKM $M^3$, we define the following open subsets:
\begin{eqnarray}
\mathcal{U}_1=\{p\in M^3:h\neq 0~~ \mbox{in a neighbourhood of}~~ p\},\nonumber\\
\mathcal{U}_1=\{p\in M^3:h=0~~ \mbox{in a neighbourhood of}~~ p\}.\nonumber
\end{eqnarray}
 Then, $\mathcal{U}_1\cup\mathcal{U}_2$ is an open and dense subset of $M^3$ and there exists a local orthonormal basis $\{e,\varphi e,\zeta\}$ of three smooth unit eigenvectors of $h$ for any point $p\in\mathcal{U}_1\cup\mathcal{U}_2$. On $\mathcal{U}_1$ we may set $he_1=\vartheta e_1$ and $he_2=-\vartheta e_2$, where $\vartheta$ is a positive function.
\begin{lemma}\cite{28c}\label{l6}
On $\mathcal{U}_1$ we have
\begin{eqnarray}
&\nabla_\zeta\zeta=0,~~~~~~~~~~~~~~~~~~\nabla_\zeta e=a\varphi e,~~~~~~~~~~~~~~~~~~~~~~~~~\nabla_\zeta\varphi e=-ae,\nonumber\\
&\nabla_e\zeta=e-\vartheta\varphi e,~~~~~~~~~~~~\nabla_ee=-\zeta-b\varphi e,~~~~~~~~~~~\nabla_e\varphi e=\vartheta\zeta+be,\nonumber\\
&\nabla_{\varphi e}\zeta=-\vartheta e+\varphi e,~~~~~~~~\nabla_{\varphi e}e=\vartheta\zeta+c\varphi e,~~~~~~~~\nabla_{\varphi e}\varphi e=-\zeta-ce,\nonumber
\end{eqnarray}
where $a,b,c$ are smooth functions.
\end{lemma}
From Lemma \ref{l6}, the poisson brackets for $\{e,\varphi e,\zeta\}$ are as follows:
\begin{eqnarray}\label{53}
[\zeta,e]=(a+\vartheta)\varphi e-e, [e,\varphi e]=be-c\varphi e, [\varphi e,\zeta]=(a-\vartheta)e+\varphi e.
\end{eqnarray}
In view of lemma \ref{l6}, the expressions for Ricci operator are as follows:
\begin{lemma}\label{l5}
On $\mathcal{U}_1$, we have
\begin{eqnarray}
\begin{cases}
Q\zeta=-2(\vartheta^2+1)\zeta-(\varphi e(\vartheta)+2\vartheta b)e-(e(\vartheta)+2\vartheta c)\varphi e,\\
Qe=-(\varphi e(\vartheta)+2\vartheta b)\zeta-(A+2\vartheta a)e+(\zeta(\vartheta)+2\vartheta)\varphi e,\\
Q\varphi e=-(e(\vartheta)+2\vartheta c)\zeta+(\zeta(\vartheta)+2\vartheta)e-(A-2\vartheta a)\varphi e,
\end{cases}
\end{eqnarray}
where $A=e(c)+\varphi e(b)+b^2+c^2+2$.
\end{lemma}
Suppose that the non-trivial potential vector field of gradient ARY soliton $(\alpha\neq 0)$ is orthogonal to the Reeb vector field $\zeta$, then we can write $V=f_1e+f_2\varphi e$, where $f_1,f_2$ are smooth functions. Replacing $\Omega_{1}$ by $\zeta$ in (\ref{a1}) and making use of Lemmas \ref{l6} and \ref{l5}, we get
\begin{eqnarray}\label{54}
\begin{cases}
-2\alpha(\vartheta^2+1)=\lambda-\frac{\beta r}{2},\\
\varphi e(\vartheta)=-2b\vartheta,\\
e(\vartheta)=-2c\vartheta.
\end{cases}
\end{eqnarray}
Similarly, taking $\Omega_{1}=e$ in (\ref{a1}) gives
\begin{eqnarray}\label{55}
\begin{cases}
\vartheta f_1-f_2=0,\\
e(f_1)+bf_2-\alpha(A+2a\vartheta)=\lambda-\frac{\beta r}{2},\\
e(f_2)-bf_1+\alpha(\zeta(\vartheta)+2\vartheta)=0.
\end{cases}
\end{eqnarray}
Also for $\Omega_{1}=\varphi e$, we obtain
\begin{eqnarray}\label{56}
\begin{cases}
\vartheta f_1-f_2=0,\\
\varphi e(f_1)-cf_2+\alpha(\zeta(\vartheta)+2\vartheta)=0,\\
cf_1+\varphi e(f_2)-\alpha(A-2a\vartheta)=\lambda-\frac{\beta r}{2}.
\end{cases}
\end{eqnarray}
Comparing the first arguments of (\ref{55}) and (\ref{56}), we see that $(\vartheta^2-1)f_1=0$. If $f_1=0$, then from first argument of (\ref{56}) we get $f_2=0$, which further implies $V=0$, a contradiction. Therefore, we must have $\vartheta=1$. In consequence of this in second and third statement of (\ref{54}) yields $b=c=0$ and first argument of (\ref{55}) gives $f_1=f_2$. \\
Combining the first equation of (\ref{54}) with second eqn. (\ref{55}) and third eqn. (\ref{56}), then making use of the fact that $\vartheta=1$ yields
\begin{eqnarray}\label{57}
e(f_1)-\varphi e(f_1)=cf_1-bf_1+4\alpha a=0,
\end{eqnarray}
where we use $f_1=f_2$. Similarly from (\ref{55}) and (\ref{56}), one can get 
\begin{eqnarray}\label{58}
\varphi e(f_1)-cf_1=e(f_1)-bf_1.
\end{eqnarray}
Making use of (\ref{57}) and (\ref{58}), together with $b=c=0$ gives $a=0$. In consequence, Eq. (\ref{53}) becomes
$$[\zeta,e]=\varphi e-e,~~~ [e,\varphi e]=0,~~~ [\varphi e,\zeta]=-e+\varphi e.$$
Using Milnor's result \cite{M} we can conclude that $M^3$ is locally isometric to a non-unimodular Lie group with a left-invariant almost Kenmotsu structure. Moreover, it is obvious that $\nabla_\zeta h=0$ and it is conformally flat with constant scalar curvature $r=-8$. Now making use of Wang's result \cite{Wa} which state that, \textit{``An almost Kenmotsu 3-manifold satisfying $\nabla_\zeta h=0$ is conformally flat with constant scalar curvature if and only if it is locally isometric to either the hyperbolic space $\mathbb{H}^3(-1)$ or the Riemannian product $\mathbb{H}^2(-4)\times\mathbb{R}$"} we can state the following:
\begin{theorem}
If a 3-dimensional non-Kenmotsu AKM admits a gradient ARYS $(\alpha\neq 0)$ whose non-trivial potential vector field is orthogonal to the Reeb vector field, then it is locally isometric to the Riemannian product $\mathbb{H}^2(-4)\times\mathbb{R}$.
\end{theorem}

\section*{Acknowledgments}

The first author is thankful to the Department of Science and Technology, New Delhi, India for financial support in the form of INSPIRE Fellowship (DST/INSPIRE Fellowship/2018/IF180830).

\end{document}